\documentclass{amsproc}
\usepackage{amssymb}
\usepackage{graphicx}
\usepackage{amscd}
\usepackage{amsmath}

\theoremstyle{plain}

\newtheorem{definition}{Definition}

\newtheorem{lemma}{Lemma}

\newtheorem{remark}{Remark}

\newtheorem{theorem}{Theorem}
\numberwithin{equation}{section}

\begin{document}

\title{Boundary Control method and De Branges spaces. Schr\"odinger equation, Dirac system and Discrete Schr\"odinger operator.}
\author{ A. S. Mikhaylov}
\address{St. Petersburg   Department   of   V.A. Steklov    Institute   of   Mathematics
of   the   Russian   Academy   of   Sciences, 7, Fontanka, 191023
St. Petersburg, Russia and Saint Petersburg State University,
St.Petersburg State University, 7/9 Universitetskaya nab., St.
Petersburg, 199034 Russia.} \email{a.mikhaylov@spbu.ru}

\author{ V. S. Mikhaylov}
\address{St.Petersburg   Department   of   V.A.Steklov    Institute   of   Mathematics
of   the   Russian   Academy   of   Sciences, 7, Fontanka, 191023
St. Petersburg, Russia and Saint Petersburg State University,
St.Petersburg State University, 7/9 Universitetskaya nab., St.
Petersburg, 199034 Russia.} \email{v.mikhaylov@spbu.ru}

\keywords{inverse problem, Boundary Control method, De Branges
method, Schr\"odinger operator, Dirac system, discrete
Schr\"odinger operator}

\maketitle

\noindent {\bf Abstract.} In the framework of the application of
the Boundary Control method to solving the inverse dynamical
problems for the one-dimensional Schr\"odinger and Dirac operators
on the half-line and semi-infinite discrete Schr\"odinger
operator, we establish the connections with the method of De
Branges: for each of the system we construct the De Branges space
and give a natural dynamical interpretation of all its
ingredients: the set of function the De Brange space consists of,
the scalar product, the reproducing kernel.

\section{Introduction}

In \cite{AM,BM_Sh} the authors attempted to look at different
approaches to inverse problems for one-dimensional systems from
one (dynamical) point of view. It happens that Gelfand-Levitan
\cite{GL}, Krein \cite{Kr2} , Simon \cite{BS1} and Remling
\cite{R1} equations can be derived within the framework of the
Boundary Control method. At the same time all the ingredients of
corresponding equations has their dynamical counterparts.

In \cite{R1,R2} the author answering the questions posed by Simon
in \cite{BS1,BSFG2}, used the De Branges method and De Branges
spaces. In \cite{AM} the authors have shown that the equations
derived by Remling \cite{R2} are in fact Krein equations and they
have clear dynamical interpretation. In the present paper we would
like to elaborate this observation: in fact the link between
Boundary Control method and De Branges method are much deeper. In
our approach we deal with the dynamical systems with boundary
control. Fixing time $T$ we take the set of the states of the
system at this time (the reachable set), taking the Fourier image
of this set we get the new space. We equip this space with the
norm generated by so-called \emph{connecting operator} to get a
Hilbert space of analytic functions. Then we construct the
reproducing kernel in this space by solving the Krein equation. We
develop this approach on the basis of three systems: Schrodinger
equation and Dirac system on the half-line and semi-infinite
discrete Schrodinger operator.

In the second section we provide all necessary information on De
Branges spaces following to \cite{R2} and \cite{RR}. In the third
section we deal with the Schr\"odinger operator on the half-line,
the forth and fifth sections are devoted to the Dirac operator on
the half line and the semi-infinite discrete Schr\"odinger
operator. For each operator we consider the dynamical settings of
the inverse problem, introduce the dynamical inverse data and
operators of the BC method. Then for each dynamical problem we
introduce special spaces which (as we prove) will be the De
Branges spaces.

\section{De Branges spaces}

Here we provide the information on De Branges spaces according to
\cite{R2,RR}. We call entire function $E:\mathbb{C}\mapsto
\mathbb{C}$ a \emph{Hermite-Biehler function} if
$|E(z)|>|E(\overline z)|$ for $z\in \mathbb{C}_+$. Let
$F^\#(z)=\overline{F(\overline{z})}$. The Hardy space $H_2$ is
defined by: $f\in H_2$ if $f$ is holomorphic in $\mathbb{C}^+$ and
$\sup_{y>0}\int_{-\infty}^\infty|f(x+iy)|^2\,dx<\infty$. Then De
Branges space $B(E)$ consists of entire functions such that:
\begin{equation*}
B(E):=\left\{F:\mathbb{C}\mapsto \mathbb{C},\,F \text{ entire},
\int_{\mathbb{R}}\left|\frac{F(\lambda)}{E(\lambda)}\right|^2\,d\lambda<\infty,\,\frac{F}{E},\frac{F^\#}{E}\in
H_2\right\}.
\end{equation*}
The space $B(E)$ with the scalar product
\begin{equation*}
[F,G]_{B(E)}=\frac{1}{\pi}\int_{\mathbb{R}}\overline{ F(\lambda)}
G(\lambda)\frac{d\lambda}{|E(\lambda)|^2}
\end{equation*}
is a Hilbert space. For any $z\in \mathbb{C}$ the
\emph{reproducing kernel} is introduced by
\begin{equation}
\label{repr_ker} J_z(\xi):=\frac{\overline{E(z)}E(\xi)-E(\overline
z)\overline{E(\overline \xi)}}{2i(\overline z-\xi)}
\end{equation}
Then
\begin{equation*}
F(z)=[J_z,F]_{B(E)}=\frac{1}{\pi}\int_{\mathbb{R}}\overline{J_z(\lambda)}
G(\lambda)\frac{d\lambda}{|E(\lambda)|^2}
\end{equation*}
We observe that a Hermite-Biehler function $E(\lambda)$ defines
$J_z$ by (\ref{repr_ker}). The converse is also true
\cite{DMcK,DBr}:
\begin{theorem}
\label{TeorDB} Let $X$ be a Hilbert space of entire functions with
reproducing kernel such that
\begin{itemize}
\item[1)] For any $\omega\in \mathbb{C}$ point evaluation is a
bounded functional, i.e. $|f(\omega)|\leqslant C\|f\|_X$.

\item[2)] if $f\in X$ then $f^\#\in X$ and $\|f\|_X=\|f^\#\|_X$

\item[3)] if $f\in X$ and $\omega\in \mathbb{C}$ such that
$f(\omega)=0$, then $\frac{z-\overline{\omega}}{z-\omega}f(z)\in
X$ and
$\left\|\frac{z-\overline{\omega}}{z-\omega}f(z)\right\|_{X}=\|f\|_{X}$.
\end{itemize}
then $X$ is a De Branges space based on the function
\begin{equation*}
E(z)=\sqrt{\pi}(1-iz)J_i(z)\|J_i\|_X^{-1}.
\end{equation*}
where $J_z$ is a reproducing kernel.
\end{theorem}


\section{Schr\"odinger equation on the half-line}

For the potential $q\in L_{1,\,loc}(\mathbb{R}_+)$ we consider the
Schr\"odinger operator on the half-line $H=-\partial^2_x+q$ on
$L_2(0,\infty)$ with Dirichlet boundary condition $\phi(0)=0$. For
$z\in \mathbb{C}$ consider the solution
\begin{equation}
\label{Schr} \left\{
\begin{array}l
-\varphi''(x)+q(x)\varphi(x)=z\varphi(x),\\
\varphi(0,z)=0,\,\, \varphi(0,z)=1.
\end{array}
\right.
\end{equation}
We fix $N\in \mathbb{R}_+$ and show that the function
$E(z):=\varphi(N,z)+i\varphi'(N,z)$ is a Hermite-Biehler function.
First we observe that $\varphi(N,\overline
z)=\overline{\varphi(N,z)}$ and $\varphi'(N,\overline
z)=\overline{\varphi'(N,z)}$, and consider (\ref{repr_ker}):
\begin{eqnarray}
J_z(\xi)=\frac{\left(\overline{\varphi(z)}-i\overline{\varphi'(z)}\right)\left(\varphi(\xi)+i\varphi'(\xi)\right)-\left(\varphi(\overline
z)+i\varphi'(\overline
z)\right)\overline{\left(\varphi(\overline\xi)+i\varphi'(\overline\xi)\right)}}
{2i(\overline z-\xi)}\notag\\
=\frac{\left(\overline{\varphi(z)}-i\overline{\varphi'(z)}\right)\left(\varphi(\xi)+i\varphi'(\xi)\right)-\left(\overline{\varphi(z)}+i\overline{\varphi'(z)}\right)\left(\varphi(\xi)-i\varphi'(\xi)\right)}{2i(\overline
z-\xi)}\notag\\
=\frac{-\overline{\varphi'(z)}\varphi(\xi)+\overline{\varphi(z)}\varphi'(\xi)}{\overline
z-\xi} \label{J_sch}
\end{eqnarray}
We take two points $z,\xi \in \mathbb{C}$ and consider
\begin{eqnarray*}
-\overline{\varphi''(x)}+q(x)\overline{\varphi(x)}=\overline
z\overline{\varphi(x)},\\
-\varphi''(x)+q(x)\varphi(x)=\xi\varphi(x)
\end{eqnarray*}
multiply the first equation by $u(\xi)$, multiply the second by
$\overline u(z)$ and subtract to get
\begin{equation*}
-\overline{\varphi''(z)}\varphi(\xi)+\varphi''(\xi)\overline{\varphi(z)}=(\overline
z-\xi)\overline{\varphi(z)}\varphi(\xi)
\end{equation*}
We integrate the above equality from $0$ to $N$ and integrate by
parts to get :
\begin{eqnarray}
(\overline
z-\xi)\int_0^N\overline{\varphi(x,z)}\varphi(x,\xi)\,dx=\left.\left[-\overline{\varphi'(x,z)}\varphi(x,\xi)+\varphi'(x,\xi)\overline{\varphi(x,z)}\right]\right|_{x=0}^{x=N}\notag\\
=-\overline{\varphi'(N,z)}\varphi(N,\xi)+\varphi'(N,\xi)\overline{\varphi(N,z)}\label{in_sch}
\end{eqnarray}
Comparing (\ref{J_sch}) and (\ref{in_sch}) we see that
\begin{equation*}
J_z(\xi)=\int_0^N\overline{\varphi(x,z)}\varphi(x,\xi)\,dx
\end{equation*}
Taking $\xi=z$:
\begin{equation}
\label{E_ineq}
0<\int_0^N\|\varphi(x,z)|^2\,dx=J_z(z)=\frac{\left|E(z)\right|^2-\left|E(\overline
z)\right|^2}{2i(-2\operatorname{Im}z)}=\frac{\left|E(z)\right|^2-\left|E(\overline
z)\right|^2}{4\operatorname{Im}z}
\end{equation}
which proves $E$ to be a a Hermite-Biehler function. Thus one can
define the De Branges space $\widehat B^N_S$ based on the function
$E$. The De Branges theory \cite{DBr} says that every De Branges
space corresponds to a certain canonical system, and provides the
procedure of recovering this system from the space (essentially
from the function $E$). So, once one have in hands the space
$\widehat B^N_S$, and we know that this space comes from
Schr\"odinger equation (special case of canonical system), it is
reasonable to pose question of recovering the potential $q$ (see
\cite{R1,R2}). Below we construct the De Branges space of
Schr\"odinger operator using the dynamical approach. And it will
be explained that "inverse problem", i.e. recovering of the system
from the De Branges space is equivalent to Boundary Control method
\cite{BIP07,AM,BM_Sh}.

It is known \cite{LeSa} that there exist a spectral measure
$d\rho(\lambda)$, such that for all $f,g\in L^2(\mathbb{R}_+)$ the
Parseval identity holds:
\begin{equation}
\int_0^\infty {f(x)}g(x)\,dx=\int_{-\infty}^\infty
{(Ff)(\lambda)}(Fg)(\lambda)\,d\rho(\lambda), \label{Fourier_int_1}\\
\end{equation}
where $F: L_2(\mathbb{R}_+)\mapsto L_{2,\,\rho}(\mathbb{R})$ is a
Fourier transformation:
\begin{eqnarray}
(Ff)(\lambda)=\int_0^\infty
f(x)\varphi(x,\lambda)\,dx\label{Fourier_int_2}\\
f(x)=\int_{-\infty}^\infty
(Ff)(\lambda)\varphi(x,\lambda)\,d\rho(\lambda)\notag.
\end{eqnarray}

For the same potential $q$ we consider the initial boundary value
problem for the 1d wave equation on the half line:
\begin{equation}
\label{wave_eqn} \left\{
\begin{array}l
u_{tt}(x,t)-u_{xx}(x,t)+q(x)u(x,t)=0, \quad x>0,\ t>0,\\
u(x,0)=u_t(x,0)=0,\ u(0,t)=f(t).
\end{array}
\right.
\end{equation}
where $f$ is an arbitrary $L^2_{loc}\left(
\mathbb{R}_{+},\mathbb{C}\right) $ function referred to as a
\emph{boundary control}. The following representation
\cite{AM,BM_Sh} for $u^f$ holds:
\begin{equation}
\label{wave_eqn_sol} u^f(x,t)=\left\{\begin{array}l
f(t-x)+\int_x^tw(x,s)f(t-s)\,ds, \quad x \leq t,\\
0, \quad x > t.\end{array}\right .
\end{equation}
Let $\mathcal{F}^T:=L^2(0,T;\mathbb{C})$ with the scalar product
$\left(f,g\right)_{\mathcal{F}^T}=\int_0^T
\overline{f(t)}{g(t)}\,dt$ be the outer space, the \emph{space of
controls}. The dynamical Dirichlet-to-Neumann map
$R^T:\mathcal{F}^T\mapsto \mathcal{F}^T$ for the system
(\ref{wave_eqn}) is defined by
\begin{equation}
\label{Response_sa} (R^Tf)(t)=u_{x}^f(0,t), \ t \in (0,T),
\end{equation}
with the domain $\{ f \in C^2([0,T];\mathbb{C}):\;
f(0)=f'(0)=0\}.$ According to (\ref{wave_eqn_sol}) it has a
representation
\begin{equation}
\label{react_rep} (R^Tf)(t)=-f'(t)+\int_0^tr(s)f(t-s)\,ds,\quad
r(s)=w_x(0,s).
\end{equation}
The wave, generated by (\ref{wave_eqn}) propagate with unite
velocity, that is why the natural setting of the dynamical inverse
problem \cite{AM,BM_Sh} is to recover $q(x),$ $x\in (0,T)$ from
$R^{2T},$ or what is equivalent, from $r(t),$ $t\in (0,2T).$

Introduce the inner space, the \emph{space of states}
$\mathcal{H}^T=L_2(0,T;\mathbb{C})$ with the scalar product
$\left(a,b\right)_{\mathcal{H}^T}=\int_0^T
\overline{a(t)}{b(t)}\,dt$ and a \emph{control operator}
\begin{equation*}
W^T: \mathcal{F}^T\mapsto \mathcal{H}^T,\quad W^Tf:=u^f(\cdot,T).
\end{equation*}
Notice that for the equation (\ref{wave_eqn}) it is natural to
consider \cite{AM,BM_Sh} the real controls (and, consequently, the
real space of states), but all the results are trivially
generalized to the case of complex $\mathcal{F}^T,$
$\mathcal{H}^T$. Everywhere below, unless it is mentioned, we
consider only real controls. The following statement is valid:
\begin{theorem}
Control operator $W^T$ is an isomorphism.
\end{theorem}

The solution $u^f$ to (\ref{wave_eqn}) admits the spectral
representation \cite{AM} at fixed time $T$:
\begin{equation}
\label{wave_repr}
W^Tf:=u^f(x,T)=\int_{-\infty}^\infty\int_0^T\frac{\sin{\sqrt{\lambda}s}}{\sqrt{\lambda}}f(T-s)\,ds\,\varphi(x,\lambda)d\rho(\lambda)
\end{equation}
We take a Fourier transform (\ref{Fourier_int_2}) of a state
$u^f(\cdot,T),$ generated by a control $f$, for $\mu\in
\mathbb{R}$ we get:
\begin{eqnarray}
\label{DB_Sch}\widehat u^f(\mu,T)=\int_{-\infty}^\infty
u^f(x,T)\varphi(x,\mu)\,dx\\=\int_{-\infty}^\infty\int_{-\infty}^\infty\int_0^T\frac{\sin{\sqrt{\lambda}s}}{\sqrt{\lambda}}f(T-s)\,ds\,\varphi(x,\lambda)d\rho(\lambda)\varphi(x,\mu)\,dx
=\int_0^T\frac{\sin{\sqrt{\mu}s}}{\sqrt{\mu}}f(T-s)\,ds.\notag
\end{eqnarray}
Since $\int_{-\infty}^\infty u^f(x,T)\varphi(x,\mu)\,dx$ is
analytic in $\mathbb{C}$, we can continue $u^f(\mu,T)$ from
$\mathbb{R}$, to get
\begin{eqnarray}
\label{DB_Sch1}\widehat
u^f(\mu,T)=\int_0^T\frac{\sin{\sqrt{\mu}s}}{\sqrt{\mu}}f(T-s)\,ds,\quad
\mu\in \mathbb{C}.
\end{eqnarray}

We introduce the space of the Fourier images of states of the
dynamical system (\ref{wave_eqn}) at time $T$ (controls are real
here):
\begin{equation*}
B_S^T:=\left\{\widehat u^f(\mu,T)\,|\, f\in \mathcal{F}^T\right\}.
\end{equation*}
Which we put as a definition of De Brange space. Bearing in mind
(\ref{DB_Sch1}), we get
\begin{equation}
\label{DB2}
B_S^T=\left\{\int_0^T\frac{\sin{\sqrt{\mu}s}}{\sqrt{\mu}}f(T-s)\,ds\,\Big|\,
f\in \mathcal{F}^T\right\}
\end{equation}
In \cite{R1,R2}, the author have shown that $B^T_S$ (precisely
(\ref{DB2})) is a De Branges space. Our aim will be to show the
same using the dynamical approach.

We introduce the \emph{connecting operator}
$C^T:\mathcal{F}^T\mapsto \mathcal{F}^T$ using the quadratic form:
\begin{equation}
\label{CT_def}
\left(C^Tf,g\right)_{\mathcal{F}^T}=\left({W^Tf},W^Tg\right)_{\mathcal{H}^T},\quad
C^T=\left(W^T\right)^*W^T.
\end{equation}
The connecting operator is an isomorphism in $\mathcal{F}^T$,
\cite{AM,BM_Sh}.


We can evaluate using the Parseval identity (\ref{Fourier_int_1})
and definition of $C^T$:
\begin{equation*}
\left(C^Tf,g\right)_{\mathcal{F}^T}=\left({u^f(\cdot,T)},u^g(\cdot,T)\right)_{\mathcal{H}^T}=\int_{-\infty}^\infty
\overline{\widehat u^f(\mu,T)}\widehat u^g(\mu,T)\,d\rho(\mu)
\end{equation*}
then we use (\ref{DB_Sch}), which yields:
\begin{eqnarray}
\left(C^Tf,g\right)_{\mathcal{F}^T}= \int_{-\infty}^\infty
\overline{\int_0^T\frac{\sin{\sqrt{\mu}s}}{\sqrt{\mu}}f(T-s)\,ds}
\int_0^T\frac{\sin{\sqrt{\mu}t}}{\sqrt{\mu}}g(T-t)\,dt\,d\rho(\mu)\notag\\
= \int_{-\infty}^\infty \overline{F(\mu)}G(\mu)\,d\rho(\mu),\label{CT_Sh_spectr}\\
F(\mu)=\int_0^T\frac{\sin{\sqrt{\mu}s}}{\sqrt{\mu}}f(T-s)\,ds,\,\,
G(\mu)=\int_0^T\frac{\sin{\sqrt{\mu}s}}{\sqrt{\mu}}g(T-s)\,ds.\label{FG}
\end{eqnarray}
Then for the functions $F,G\in B^T_S$ having the representations
(\ref{FG}), we can introduce the scalar product in $B^T_S$ by
\begin{equation}
\label{Scal_Sh}
[F,G]_{B^T_S}:=\left(C^Tf,g\right)_{\mathcal{F}^T}.
\end{equation}
The fact that $C^T$ is an isomorphism implies that the space
$B^T_S$ equipped with the norm, generated by this scalar product
is a Hilbert space.

For positive $N$ we can prescribe a self-adjoint boundary
condition at $x=N:$
\begin{equation}
\label{Schr_D} \left\{
\begin{array}l
-\varphi''(x)+q(x)\varphi(x)=z\varphi(x),\\
\varphi(0,z)=0,\,\, \alpha\varphi(N,z)+\beta\varphi'(N,z)=0.
\end{array}
\right.
\end{equation}
The (discrete) measure corresponding to (\ref{Schr_D}) we denote
by $d\rho_N(\lambda)$.
\begin{remark}
Due to the finite speed of wave propagation in the dynamical
system (\ref{wave_eqn}), equal to one, in all formulaes starting
from (\ref{wave_repr}), we can substitute the measure
$d\rho(\lambda)$ by any measure $d\rho_N(\lambda)$ with
$N\geqslant T.$ And consequently,
\begin{equation*}
[F,G]_{B^T_S}= \int_{-\infty}^\infty
\overline{F(\mu)}G(\mu)\,d\rho(\mu)= \int_{-\infty}^\infty
\overline{F(\mu)}G(\mu)\,d\rho_N(\mu).
\end{equation*}
\end{remark}

It is a crucial fact in BC method that $C^T$ admits the
representation in terms of the inverse data \cite{AM,BM_Sh}:
\begin{theorem}
Control operator $C^T$ admits the representation in terms of the
dynamical data
\begin{equation}
\label{r-c2} ({C}^T  f)(t)=f(t)+\int_0^T c^T(t,s)f(s)\,ds\,, \
0<t<T\,,
\end{equation}
where
\begin{equation}
\label{c_t} c^T(t,s)=[p(2T-t-s)-p(t-s)],\quad
p(t):=\frac{1}{2}\int_0^{|t|} r\left(s\right)\,ds
\end{equation}
and spectral data:
\begin{equation*}
\left(C^Tf\right)(x)=\int_0^T C(x,y)f(y)\,dy,\,\,
C(x,y)=\int_{-\infty}^\infty
\frac{\sin{\sqrt{\mu}(T-x)}}{\sqrt{\mu}}\frac{\sin{\sqrt{\mu}(T-y)}}{\sqrt{\mu}}\,d\rho(\lambda),
\end{equation*}
where the action of generalized kernel C(x,y) is defined by
(\ref{CT_Sh_spectr}).
\end{theorem}
Let $J_z$ be the reproducing kernel in $B^T_S$, the latter means
that for all $F\in B^T_S$ the following should hold for $z\in
\mathbb{C}$:
\begin{equation}
\label{Point_ev} [J_z,F]_{B^T_S}=F(z).
\end{equation}
Let
$F(\mu)=\int_0^T\frac{\sin{\sqrt{\mu}s}}{\sqrt{\mu}}f(T-s)\,ds$.
We look for $J_z$ in the form:
\begin{equation}
\label{J_sch_repr}
J_z(\mu)=\int_0^T\frac{\sin{\sqrt{\mu}(T-s)}}{\sqrt{\mu}}j_z(s)\,ds.
\end{equation}
Evaluating l.h.s. of (\ref{Point_ev}) using (\ref{Scal_Sh}) and
r.h.s. of (\ref{Point_ev}) using representation of $F$ and fact
that $f$ is real, we arrive at
\begin{equation*}
\left(C^Tj_z,f\right)_{\mathcal{F}^T}=\int_0^T\frac{\sin{\sqrt{z}(T-s)}}{\sqrt{z}}f(s)\,ds,
\end{equation*}
which immediately yields the following Krein equation on $j_z:$
\begin{equation}
\label{CT_eq_sch}
\left(C^Tj_z\right)(s)=\overline{{\frac{\sin{\sqrt{z}(T-s)}}{\sqrt{z}}}},\quad
s\in (0,T).
\end{equation}
Notice that (\ref{CT_eq_sch}) has a unique solution due to the
fact that $C^T$ is an isomorphism.

Let us set up the \emph{special control problem}: for $z\in
\mathbb{C}$ to find a (complex-valued) control $f_z\in
L_2(0,T;\mathbb{C})$ such that $W^Tf_z=\varphi(x,z)$, $x\in
(0,T)$. Notice that only here we deal with complex-valued
controls.
\begin{lemma}
The solution of the special control problem can be found as a
unique solution to the Krein equation (\ref{CT_eq_sch}).
\end{lemma}
\begin{proof}
We take the equality
\begin{equation*}
W^Tf_z=\varphi(x,z),\quad x\in (0,T),
\end{equation*}
and multiply it in $\mathcal{H}^T$ by $W^T g$, $g\in
\mathcal{F}^T$. As result we get that
\begin{equation}
\label{K1}
\left(W^Tf_z,W^Tg\right)_{\mathcal{H}^T}=\left(\varphi(\cdot,z),W^Tg\right)_{\mathcal{H}^T}=\overline{\int_0^T\overline{\left(W^Tg\right)(x)}\varphi(x,z)\,dx
}.
\end{equation}
The definition of $C^T$ (\ref{CT_def}) and spectral representation
(\ref{wave_repr}) transform (\ref{K1}) to:
\begin{equation}
\label{K2} \left(C^Tf_z,g\right)_{\mathcal{F}^T}=
\overline{\int_{0}^T\int_{-\infty}^\infty\int_0^T\overline{\frac{\sin{\sqrt{\lambda}s}}{\sqrt{\lambda}}g(T-s)}\,ds\,\varphi(x,\lambda)d\rho(\lambda)\varphi(x,z)\,dx},
\end{equation}
From (\ref{K2}) as in the proof of (\ref{DB_Sch}) we deduce that
\begin{equation*}
\left(C^Tf_z,g\right)_{\mathcal{F}^T}=\int_0^T\frac{\sin{\sqrt{z}(T-s)}}{\sqrt{z}}g(s)\,ds,
\end{equation*}
which proves the statement.
\end{proof}
We notice that initially the Krein equations were derived using
purely dynamical approach (see \cite{AM,BM_Sh}).

So, having constructed reproducing kernel $J_z(\lambda)$ from
Krein equation (\ref{CT_eq_sch}) and convolution formula
(\ref{J_sch_repr}), we can recover $E(\lambda)$ using Theorem
\ref{TeorDB}, all condition of which are clearly satisfied.

We show that the the fact that $E(\lambda)$ is a Hermite-Biehler
function follows from the positivity of $C^T$. Indeed, as it
follows from (\ref{E_ineq}),
\begin{eqnarray*}
\frac{\left|E(z)\right|^2-\left|E(\overline
z)\right|^2}{4\operatorname{Im}z}=J_z(z)=\int_0^T\frac{\sin{\sqrt{z}(T-s)}}{\sqrt{z}}j_z(s)\,ds\\
=\int_0^T\frac{\sin{\sqrt{z}(T-s)}}{\sqrt{z}}\left(\left(C^T\right)^{-1}\overline{{\frac{\sin{\sqrt{z}(T-\cdot)}}{\sqrt{z}}
}}\right)\left(s\right)\,ds\\
=\left(\left(\left(C^T\right)^{-1}\right)^*\overline{\frac{\sin{\sqrt{z}(T-s)}}{\sqrt{z}}},\overline{\frac{\sin{\sqrt{z}(T-s)}}{\sqrt{z}}}
\right)_{\mathcal{F}^T}>0,
\end{eqnarray*}
where the last inequality follows from the positivity of $C^T$.

If we know the De Branges space $B^T_S$, we can recover the
potential $q(x)$, $x\in (0,T)$ using the general theory of
canonical systems  \cite{DBr,RR,R1,R2}, or using the  Boundary
Control method (we need to know the operator $C^T$ only!). For the
details see \cite{AM,BM_Sh}.

\section{Dirac system on the half-line}

We consider the operator of the Dirac system on the half-line:
introduce the matrix $J:=\begin{pmatrix} 0&1\\-1&0\end{pmatrix}$
and  a matrix \emph{potential} $V=\begin{pmatrix} p&q\\q&-p
\end{pmatrix}$, $p,q\in C^1_{loc}(R_+)$. We set
$\mathcal{D}:=J+V$ on
$L_2(\mathbb{R}_{+},\mathbb{R}^2)\ni\Phi=\begin{pmatrix}\Phi_1\\\Phi_2\end{pmatrix}$
with Dirichlet condition $\Phi_1(0)=0$.

Let $\theta(x,z)=\begin{pmatrix}\theta^1\\ \theta^2\end{pmatrix}$
be the solution to the following Cauchy problem
\begin{equation}
\label{spec_sol} \left\{\begin{array}l
J\theta_x+V\theta=z\theta, \quad x>0, \\
\theta_1(0,z)=0, \quad \theta_2(0,z)=1.
\end{array}\right.
\end{equation}
We fix $N\in \mathbb{R}_+$ and show that
$E(z):=\theta^1(N,z)-i\theta^2(N,z)$ is a Hermite-Biehler
function. Let us evaluate (\ref{repr_ker}), counting that
$\overline{\theta(x,z)}=\theta(\overline z)$:
\begin{eqnarray}
J_z(\xi)=\frac{\left(\overline{\theta^1(z)}+i\overline
{\theta^2(z)}\right)\left(\theta^1(\xi)-i\theta^2(\xi)\right)-\left(\overline{\theta^1(z)}-i\overline
{\theta^2(z)}\right)\left(\theta^1(\xi)+i\theta^2(\xi)\right)}{2i(\overline
z-\xi)}\notag\\
=\frac{\overline{\theta^2(z)}\theta^1(\xi)-\overline{\theta^1(z)}\theta^2(\xi)}{\overline
z-\xi}\label{repr_dir}
\end{eqnarray}
We take points $z,\xi \in \mathbb{C}$ and consider
\begin{eqnarray*}
J\overline{\theta'(x,z)}+V(x)\overline{\theta(x,z)}=\overline{
z}\overline{\theta(x,z)},\\
J\theta'(x,\xi)+V(x)\theta(x,\xi)=\xi\theta(x,\xi).
\end{eqnarray*}
multiply the first equation by $\theta(\xi)$, multiply the second
by $\overline \theta(z)$ and subtract from the first to get
\begin{equation*}
\left(J\overline
\theta'(z),\theta(\xi)\right)_{\mathbb{R}^2}-\left(J\overline
\theta'(\xi),\overline{\theta(z)}\right)_{\mathbb{R}^2}=\left(\overline
z-\xi\right)\left(\overline
\theta(z),\theta(\xi)\right)_{\mathbb{R}^2}.
\end{equation*}
We integrate the latter equality from zero to $N$ and evaluate:
\begin{eqnarray*}
\left(\overline z-\xi\right)\int_{0}^N\left(\overline{
\theta(z)},\theta(\xi)\right)_{\mathbb{R}^2}\,dx=\int_0^N\left[\left(\overline
{\theta'_2(z)}\theta_1(\xi)-\overline
{\theta'_1(z)}\theta_2(\xi)\right)-\right.\\
\left.- \left( \theta'_2(\xi)\overline{\theta_1(z)}
-\theta'_1(\xi)\overline{\theta_2(z)}\right)\right]\,dx=\overline
{\theta_2(z)}\theta_1(\xi)-\overline
{\theta_1(z)}\theta_2(\xi)|_{x=N}.
\end{eqnarray*}
From here (see also (\ref{repr_dir})) follows that
\begin{equation*}
J_z(\xi)=\int_0^N\left(\overline{\theta(x,z)},\theta(x,\xi)\right)_{\mathbb{R}^2}\,dx.
\end{equation*}
Taking $\xi=z$:
\begin{equation*}
0<\int_0^N|\theta(x,z)|^2\,dx=J_z(z)=\frac{\left|E(z)\right|^2-\left|E(\overline
z)\right|^2}{2i(-2\operatorname{Im}z)}=\frac{\left|E(z)\right|^2-\left|E(\overline
z)\right|^2}{4\operatorname{Im}z},
\end{equation*}
which proves $E$ to be a a Hermite-Biehler function. By this
function one can construct the De Branges space $\widehat B^N_D$.
On the contrary, having in hands this space one can use the De
Branges technique \cite{DBr,RR} to recover the canonical system
(the Dirac system or the potential matrix $V(x)$) it comes from.
Below we use the dynamical approach to construct the De Branges
space for the Dirac system.

With a Dirac operator we associate the initial boundary-value
problem
\begin{equation}
\label{d1}
\left\{
\begin{array}l
iu_t+Ju_x+Vu=0,\,\,  x>0, \quad 0<t<T,\\
u\big|_{t=0}=0, \,\, x {\geqslant} 0, \\
u_1\big|_{x=0}=f,\,\,  0{\leqslant} t {\leqslant} T,
\end{array}
\right.
\end{equation}
where $T>0$ is a final moment; $f=f(t)$ is a complex-valued
function ({\it boundary control});
$u=u^f(x,t)=\begin{pmatrix} u_1(x,t)\\
u_2(x,t) \end{pmatrix}$ is a solution. We denote the outer space
of (\ref{d1}), the set of controls by
$\widetilde{\mathcal{F}}^T:=L_2\left((0,T);\mathbb{C}\right)$ with
the scalar product
$\left(f,g\right)_{\mathcal{F}^T}=\int_0^T\overline{f(t)}{g(t)}\,dx$.
In \cite{BM_Dir} the authors proved the following
\begin{theorem}
The solution to (\ref{d1}) admits the following representation:
\begin{equation}\label{repres u^f}
u^f(x,t)=f(t-x)\begin{pmatrix} 1\\
i \end{pmatrix}+\int_x^t w(x,s)f(t-s)\,ds\,, \qquad x {\geqslant}
0,\,\,0{\leqslant} t {\leqslant} T
\end{equation}
holds with $w=\begin{pmatrix} w_1\\
w_2\end{pmatrix}$ being a vector-kernel such that $w\big|_{t<x}=
0$, $w\big|_{\Delta^T} \in C^1(\Delta^T; {\mathbb C}^2)$, and
$w_1(0,\cdot)=0$.
\end{theorem}

The \emph{response operator} $R^T:\widetilde{\mathcal{F}}^T\mapsto
\widetilde{\mathcal{F}}^T$ with the domain $\{f\in
C^2(0,T;C^2)\,|\, f(0)=0\}$, the analog of dynamical Dirichlet to
Neumann map is defined by
\begin{equation}
\label{Dirac_react} R^Tf:=u^f_2(0,t),\quad 0<t<T.
\end{equation}
from (\ref{repres u^f}) we deduce
\begin{equation}
\label{Direc_react} (Rf)(t)=if(t)+\int_0^tr(s)f(t-s)\,ds,\quad
r(s):=w_2(0,s).
\end{equation}
The speed of the wave propagation for (\ref{d1}) is equal to one,
so the natural set up of the dynamical inverse problem is to
recover $V(x)$, $x\in (0,T)$ from $R^{2T}$, or what is equivalent,
from $r(t),$ $t\in (0,2T)$.

For the vector functions $f,g\in L_2(\mathbb{R}_+, \mathbb{R}^2)$
define the Fourier transform (see \cite{LeSa})
\begin{eqnarray}
\left(F\begin{pmatrix}f_1\\
f_2\end{pmatrix}\right)(\lambda)=F(\lambda)=\int_0^\infty
f_1(x)\theta_1(x,\lambda)+f_2(x)\theta_2(x,\lambda)\,dx,\label{Dirac_FT}\\
\left(F\begin{pmatrix}g_1\\
g_2\end{pmatrix}\right)(\lambda)=G(\lambda)=\int_0^\infty
g_1(x)\theta_1(x,\lambda)+g_2(x)\theta_2(x,\lambda)\,dx\notag
\end{eqnarray}
Then there exist the measure $d\rho(\lambda)$ such that
\begin{eqnarray*}
f_1(x)=\int_{-\infty}^\infty
F(\lambda)\theta_1(x,\lambda)\,d\rho(\lambda),\,\,
f_2(x)=\int_{-\infty}^\infty F(\lambda)\theta_2(x,\lambda)\,d\rho(\lambda),\\
g_1(x)=\int_{-\infty}^\infty
G(\lambda)\theta_1(x,\lambda)\,d\rho(\lambda),\,\,
g_2(x)=\int_{-\infty}^\infty
G(\lambda)\theta_2(x,\lambda)\,d\rho(\lambda),
\end{eqnarray*}
and Parseval identity holds
\begin{eqnarray*}
\int_0^\infty f_1^2(x)+f_2^2(x)\,dx=\int_{-\infty}^\infty
F^2(\lambda)\,d\rho(\lambda),\\
\int_0^\infty f_1(x)g_1(x)+f_2(x)g_2(x)\,dx=\int_{-\infty}^\infty
F(\lambda)G(\lambda)\,d\rho(\lambda).
\end{eqnarray*}
The solution to (\ref{d1}) admits the spectral representation:
\begin{equation}
\label{Dirac_repr} u^f(x,t)=\int_{-\infty}^\infty\int_0^t
e^{i\lambda s}if(t-s)\,ds\,\theta(x,\lambda)\,d\rho(\lambda).
\end{equation}
We denote the set of states by $\mathcal{H}^T:=L_2
\left((0,T);\mathbb{C}^2\right)$ with an inner product
$(a,b)_{{{\mathcal H}^T}} := \int_{\Omega^T}\overline{a(x)} \cdot
 b(x)\,dx$, it is the \emph{inner space} of the system
(\ref{d1}). Thus for all $T>0$, $u^f(\cdot,T)\in \mathcal{H}^T$.

We define the \emph{control operator} $\widetilde W^T:
\widetilde{\mathcal{F}}^T\mapsto \mathcal{H}^T$ by
$\widetilde{W}^Tf:=u^f(\cdot,T)$ and observe (see \cite{BM_Dir})
that is not an isometry, as it easily follows from (\ref{repres
u^f}), $\widetilde{W}^T\widetilde{\mathcal F} ^T\not=
\mathcal{H}^T$. To "improve" the lack of the controllability, we
consider the auxiliary system
\begin{equation}
\label{Dir 1 opt v}
\left\{
\begin{array}l
iv_t-Jv_x-Vv=0, \,\,  0<x<T,\,\, 0<t<T\\
v\big|_{t=0}=0\,\, \\
v_1\big|_{x=0}=g,\,\, 0{\leqslant} t {\leqslant} T
\end{array}
\right.
\end{equation}
The solution $v=v^g(x,t)$ are connected with the solutions to
(\ref{d1}) by
\begin{equation*}
v^g(x,t)\,=\,\overline{u^{\bar g}(x,t)}\,.
\end{equation*}
Then we introduce the extended set of controls
$\mathcal{F}^T:=L_2((0,T);\mathbb{C}^2)$, and as a (extended)
state of Dirac system at the time $t=T$ we put
$u^f(\cdot,T)+v^g(\cdot,T)$. The new "extended" control operator
we define by $W^T:\mathcal{F}^T \mapsto \mathcal{H}^T$,
\begin{equation}
\label{W_fg} W^T\begin{pmatrix} f_1 \\
f_2\end{pmatrix}:=u^{f_1}(\cdot,T)+v^{f_2}(\cdot,T).
\end{equation}
The following statement is proved in \cite{BM_Dir}:
\begin{theorem}
The "extended" control operator $W^T$ is an isomorphism between
$\mathcal{F}^T$ and $\mathcal{H}^T$.
\end{theorem}

The spectral representation of $v^g$ is
\begin{equation}
\label{Dirac_repr_1} v^g(x,t)=\int_{-\infty}^\infty\int_0^t
e^{-i\lambda s}(-i) g(t-s)\,ds\,\theta(x,\lambda)\,d\rho(\lambda).
\end{equation}
Taking the the Fourier transform (\ref{Dirac_FT}) of
$u^f(\cdot,T)$ and $v^g(\cdot,T)$ for $\lambda\in \mathbb{R}$ we
get respectively:
\begin{eqnarray}
\left(Fu^f(\cdot,T)\right)(\lambda)=\int_0^T e^{i\lambda s}if(T-s)\,ds,\label{Fur_dir1}\\
\left(Fv^g(\cdot,T)\right)(\lambda)=-\int_0^T e^{-i\lambda s}i
g(T-s)\,ds. \label{Fur_dir2}
\end{eqnarray}
The \emph{connecting operator} $C^T: \mathcal{F}^T\mapsto
\mathcal{F}^T$ is defined by the quadratic form
\begin{equation*}
\left(C^T\begin{pmatrix} f_1 \\
f_2\end{pmatrix},\begin{pmatrix} g_1 \\
g_2\end{pmatrix}\right)_{\mathcal{F}^T}=\left(W^T\begin{pmatrix} f_1 \\
f_2\end{pmatrix},W^T\begin{pmatrix} g_1 \\
g_2\end{pmatrix}\right)_{\mathcal{H}^T}.
\end{equation*}
Notice that $C^T$ is positive isomorphism in $\mathcal{F}^T$, see
\cite{BM_Dir}.

We can evaluate making use of (\ref{Dirac_repr}),
(\ref{Dirac_repr_1}), (\ref{Fur_dir1}), (\ref{Fur_dir2}) and
Parseval identity:
\begin{eqnarray}
\label{CT_dir_sp} \left(C^T\begin{pmatrix} f_1 \\
f_2\end{pmatrix},\begin{pmatrix} g_1 \\
g_2\end{pmatrix}\right)_{\mathcal{F}^T}
=\int_{-\infty}^\infty \overline{\left(FW^T
\begin{pmatrix} f_1 \\ f_2\end{pmatrix}\right)}{\left(FW^T
\begin{pmatrix} g_1 \\ g_2\end{pmatrix}\right)}\,d\rho(\lambda)
\\=\int_{-\infty}^\infty
\overline{\int_0^T \begin{pmatrix} ie^{i\lambda(T-s)} \\
-ie^{-i\lambda(T-s)}\end{pmatrix}\begin{pmatrix} f_1(s) \\ f_2(s)\end{pmatrix}\,ds}{\int_0^T \begin{pmatrix} ie^{i\lambda(T-t)} \\
-ie^{-i\lambda(T-t)}\end{pmatrix}\begin{pmatrix} g_1(t) \\
 g_2(t)\end{pmatrix}\,ds}\, d\rho(\lambda).\notag
\end{eqnarray}
The important fact proved in \cite{BM_Dir} that $C^T$ admits a
representation in terms of inverse data:
\begin{theorem}
The control operator is represented in terms of inverse dynamial
data:
\begin{align}
\label{CT beta via R2T alpha} & \left(C^T a\right)(t)=2
a(t)+\int_0^Tc^T(t,s) a(s)\,ds\,, \qquad 0 {\leqslant} t
{\leqslant} T
\end{align}
where $c^T(t,s)$ is a matrix kernel with the elements
\begin{align}
\notag & c_{11}(t,s)(t)=-i\,[r(t-s)-\bar r(s-t)]\,, \quad
c_{12}(t,s)=-i\,\bar r(2T-t-s)\,,\\
& c_{21}(t,s)= i\,r(2T-t-s)\,,\quad c_{22}(t,s)(t)=i\,[\bar
r(t-s)-r(s-t)]\label{kernel c in detail in forward problem}\,,
\end{align}
and in terms of inverse spectral data:
\begin{equation}
\left(C^T\begin{pmatrix} f_1\\f_2\end{pmatrix}\right)(x)=\int_0^T
\left(C(x,y)\begin{pmatrix} f_1(y)\\f_2(y)\end{pmatrix}\right)\,dy
\end{equation}
where the generalized kernel of $C^T$ is given by
\begin{equation*}
C(t,s)=\int_{-\infty}^\infty \overline{c(s,\lambda)}\otimes{
c^T(t,\lambda)} \,d\rho(\lambda),\quad c(s,\lambda)=\begin{pmatrix} ie^{i\lambda(T-s)} \\
-ie^{-i\lambda(T-s)}\end{pmatrix}.
\end{equation*}
and action is given by the r.h.s. of (\ref{CT_dir_sp}).
\end{theorem}
Since $\left(Fu^f(\cdot,T)\right)(\lambda),$
$\left(Fv^g(\cdot,T)\right)(\lambda)$ are analytic in
$\mathbb{C}$, and on real line are given by (\ref{Fur_dir1}),
(\ref{Fur_dir2}) it follows that the Fourier transform of
"extended" state at time $t=T$ can be analytically continued on
$\mathbb{C}$ by the formula
\begin{equation}
\label{WT_Fur_Dir}
\left(FW^T\begin{pmatrix}f_1 \\
f_2\end{pmatrix}\right)(\lambda)=\int_0^T\left(i\begin{pmatrix}e^{i\lambda
s}\\-e^{-i\lambda
s}\end{pmatrix},\begin{pmatrix}f_1(T-s)\\f_2(T-s)\end{pmatrix}\right)\,ds,\quad
\lambda\in \mathbb{C}.
\end{equation}

We introduce the Be Branges space associated to Dirac system as a
set of Fourier transforms of of the (extended) states of the
system (\ref{d1}) at the moment $T$:
\begin{equation*}
B^T_D:=\left\{ F(\lambda)=\left(FW^T\begin{pmatrix}f_1 \\
f_2\end{pmatrix}\right)(\lambda)\Bigl|
\begin{pmatrix}f_1\\f_2\end{pmatrix}\in \mathcal{F}^T\right\}
\end{equation*}
The relation (\ref{WT_Fur_Dir}) implies that
\begin{equation}
\label{DB_Dir} B^T_D=\left\{
\int_0^T\left(i\begin{pmatrix}e^{i\lambda s}\\-e^{-i\lambda
s}\end{pmatrix},\begin{pmatrix}f_1(T-s)\\f_2(T-s)\end{pmatrix}\right)\,ds\Bigl|
\begin{pmatrix}f_1\\f_2\end{pmatrix}\in \mathcal{F}^T\right\}
\end{equation}
In $B^T_D$ we introduce the scalar product by
\begin{equation*}
[F,G]_{B^T_D}:=\left(C^T\begin{pmatrix}f_1 \\
f_2\end{pmatrix},\begin{pmatrix}g_1 \\
g_2\end{pmatrix}\right)_{\mathcal{F}^T},\quad F,G\in B^T_D.
\end{equation*}
According to (\ref{CT_dir_sp}):
\begin{eqnarray*}
[F,G]_{B^T_D}:=\int_{-\infty}^\infty
\overline{\int_0^T \begin{pmatrix} ie^{i\lambda(T-s)} \\
-ie^{-i\lambda(T-s)}\end{pmatrix}\begin{pmatrix} f_1(s) \\ f_2(s)\end{pmatrix}\,ds}{\int_0^T \begin{pmatrix} ie^{i\lambda(T-t)} \\
-ie^{-i\lambda(T-t)}\end{pmatrix}\begin{pmatrix} g_1(t) \\
 g_2(t)\end{pmatrix}\,ds}\, d\rho(\lambda)\\
 F(\lambda)=\int_0^T \begin{pmatrix} ie^{i\lambda(T-s)} \\
-ie^{-i\lambda(T-s)}\end{pmatrix}\begin{pmatrix} f_1(s) \\
f_2(s)\end{pmatrix}\,ds,\,\, G(\lambda)=\int_0^T \begin{pmatrix} ie^{i\lambda(T-t)} \\
-ie^{-i\lambda(T-t)}\end{pmatrix}\begin{pmatrix} g_1(t) \\
 g_2(t)\end{pmatrix}\,ds.
\end{eqnarray*}
Since $C^T$ is a positive isomorphism in $\mathcal{F}^T$, the
space $B^T_D$ with the norm generated by $[,]_{B^T_D}$ is a
Hilbert space.
Let $J_z(\lambda)$ be the reproducing kernel in $B^T_D$, the
latter means that
\begin{equation}
\label{Point_ev_D} [J_z,F]_{B^T_D}=F(z),\quad \forall F\in B^T_D.
\end{equation}
We will look for $J_z$ in the form:
\begin{equation}
\label{J_Dir_for}
J_z(\lambda)=\int_0^T \begin{pmatrix} ie^{i\lambda(T-s)} \\
-ie^{-i\lambda(T-s)}\end{pmatrix}\begin{pmatrix} j^z_1(s) \\
j^z_2(s)\end{pmatrix}\,ds,
\end{equation}
then from (\ref{Point_ev_D}) and definition of the scalar product
we deduce
\begin{equation*}
[J_z,F]_{B^T_D}=\left(C^T\begin{pmatrix}j^z_1 \\
j^z_2\end{pmatrix},\begin{pmatrix}f_1 \\
f_2\end{pmatrix}\right)_{\mathcal{F}^T}=F(z)=\int_0^T \begin{pmatrix} ie^{iz(T-s)} \\
-ie^{-iz(T-s)}\end{pmatrix}\begin{pmatrix} f_1(s) \\
f_2(s)\end{pmatrix}\,ds,
\end{equation*}
from where due to the arbitrariness of $f_1,f_2$ we arrive at the
following equation on $\begin{pmatrix}j^z_1\\j^z_2\end{pmatrix}$:
\begin{equation}
\label{Gener_ker2}
C^T\begin{pmatrix}j^z_1 \\
j^z_2\end{pmatrix}=\overline{\begin{pmatrix}ie^{i\lambda (T-s)}\\
-ie^{-i\lambda (T-s)}\end{pmatrix}}, \quad 0\leqslant t\leqslant
T.
\end{equation}
We emphasize that equation (\ref{Gener_ker2}) is Krein equations
and can be used for solving the inverse problem of the recovering
potential from the dynamical (or spectral) inverse data.

Let us show that $\begin{pmatrix}j^z_1\\ j^z_2\end{pmatrix}$ is a
solution to the following \emph{special control problem}.  We fix
$z\in \mathbb{C}$ and consider the control problem to find
$\begin{pmatrix}f^z_1\\ f^z_2\end{pmatrix}\in \mathcal{F}^T$ such
that
\begin{equation}
\label{Control_Dir}
W^T\begin{pmatrix}f_1^z\\
f_2^z\end{pmatrix}(\cdot,T)=\theta(\cdot,z),\quad \text{on
$(0,T)$}.
\end{equation}
Since $W^T$ is boundedly invertible, such a control $\begin{pmatrix}f_1^z\\
f_2^z\end{pmatrix}$ exists.
\begin{lemma}
The solution of the special control problem (\ref{Control_Dir})
can be found as a unique solution to the Krein equation
(\ref{Gener_ker2}).
\end{lemma}
\begin{proof}
We take the equality (\ref{Control_Dir})
and multiply it in $\mathcal{H}^T$ by $W^T\begin{pmatrix}g_1\\
g_2\end{pmatrix}$ for some $\begin{pmatrix}g_1\\
g_2\end{pmatrix}\in \mathcal{F}^T$. As a result we get that
\begin{equation}
\label{K4}
\left(W^T\begin{pmatrix}f_1^z\\
f_2^z\end{pmatrix},W^T\begin{pmatrix}g_1\\
g_2\end{pmatrix}\right)_{\mathcal{H}^T}=\left(\theta(\cdot,z),W^T\begin{pmatrix}g_1\\
g_2\end{pmatrix}\right)_{\mathcal{H}^T}=\overline{\int_0^T \overline{W^T\begin{pmatrix}g_1\\
g_2\end{pmatrix}(x)}\theta(x,z)\,dx}.
\end{equation}
The r.h.s. of (\ref{K4}) can be evaluated as (see
\ref{WT_Fur_Dir}):
\begin{eqnarray}
\overline{\int_0^T \overline{W^T\begin{pmatrix}g_1\\
g_2\end{pmatrix}(x)}\theta(x,z)\,dx}\notag\\
=\overline{\int_0^T\overline{\int_{-\infty}^\infty \int_0^T \left(
\begin{pmatrix} ie^{i\lambda s} \\ -ie^{-i\lambda s} \end{pmatrix},\begin{pmatrix}
g_1(T-s)\\g_2(T-s)\end{pmatrix}\right)\,ds\theta(x,\lambda)\,d\rho(\lambda)}\theta(x,z)\,dx}
\notag\\
=\int_0^T{\left(i\begin{pmatrix}e^{i\lambda (T-s)}\\-e^{-i\lambda
(T-s)}\end{pmatrix},\begin{pmatrix}g_1(s)\\g_2(s)\end{pmatrix}\right)}\,ds,\label{K5}
\end{eqnarray}
From (\ref{K4}), (\ref{K5}) we get the desired equation
\ref{Gener_ker2}.
\end{proof}

After we found the reproducing kernel $J_z(\lambda)$ from
(\ref{Gener_ker2}), (\ref{J_Dir_for}), we can recover $E(\lambda)$
making use of Theorem \ref{TeorDB}.

We show that the $E(\lambda)$ will be the Hermite-Biehler
function. It follows from (\ref{E_ineq}),
\begin{eqnarray*}
\frac{\left|E(z)\right|^2-\left|E(\overline
z)\right|^2}{4\operatorname{Im}z}=J_z(z)=\int_0^T \begin{pmatrix} ie^{i\lambda(T-s)} \\
-ie^{-i\lambda(T-s)}\end{pmatrix}\begin{pmatrix} j^z_1(s) \\
j^z_2(s)\end{pmatrix}\,ds\\
=\int_0^T\begin{pmatrix}ie^{i\lambda (T-s)}\\
-ie^{-i\lambda (T-s)}\end{pmatrix}\left(\left(C^T\right)^{-1}\overline{\begin{pmatrix}ie^{i\lambda (T-s)}\\
-ie^{-i\lambda (T-s)}\end{pmatrix}}\right)\left(s\right)\,ds\\
=\left(\left(\left(C^T\right)^{-1}\right)^*\overline{\begin{pmatrix}ie^{i\lambda (T-s)}\\
-ie^{-i\lambda
(T-s)}\end{pmatrix}},\overline{\begin{pmatrix}ie^{i\lambda (T-s)}\\
-ie^{-i\lambda (T-s)}\end{pmatrix}} \right)_{\mathcal{F}^T}>0,
\end{eqnarray*}
where the last inequality follows from the positivity of $C^T$.

For positive $N$ we can consider the Dirac system on $(0,N)$ with
some self-adjoint boundary condition at $x=N:$
\begin{equation}
\label{Dirac_discr} \left\{\begin{array}l
JU_x+VU=zU, \quad 0<x<N, \\
U_1(0,z)=0, \quad \alpha U_1(N,z)+\beta U_2(N,z)=0 .
\end{array}\right.
\end{equation}
The (discrete) measure corresponding to (\ref{Schr_D}) we denote
by $d\rho_N(\lambda)$.
\begin{remark}
Due to the finite speed of wave propagation in the dynamical
system (\ref{d1}), equal to one, in all formulaes starting from
spectral representation of the solution (\ref{Dirac_repr}), we can
substitute the measure $d\rho(\lambda)$ by any measure
$d\rho_N(\lambda)$ with $N\geqslant T.$ In particular
\begin{equation*}
[F,G]_{B^T_D}= \int_{-\infty}^\infty
\overline{F(\mu)}G(\mu)\,d\rho(\mu)= \int_{-\infty}^\infty
\overline{F(\mu)}G(\mu)\,d\rho_N(\mu)
\end{equation*}
\end{remark}
If we know the De Branges space $B^T_D$, we can recover the
canonical system connected with this space using the De Branges
theory \cite{DBr,RR}, or recover the Dirac system (the matrix
potential $V$) using the Boundary Control method. For the details
see \cite{BM_Dir}.

\subsection{Special case: connection between Dirac and
Schr\"odinger De Branges spaces}

We consider the system (\ref{d1}) with the special matrix
potential
\begin{equation}
\label{Spec_matr}
V=\begin{pmatrix} 0&q\\q&0
\end{pmatrix},\quad \text{$q$ is differentiable, $q(0)=0$.}
\end{equation}
We differentiate (\ref{d1}) w.r.t. $t$ and $x$ to get
\begin{eqnarray*}
iu^1_{tt}+u^2_{xt}+qu^2_t=0,\\
iu^2_{tt}-u^1_{xt}+qu^1_t=0,\\
iu^2_{tx}-u^1_{xx}+\left(qu^1\right)_x=0
\end{eqnarray*}
On introducing the special potential
\begin{equation}
\label{spec_potent} Q(x)=q_x(x)+q^2(x),
\end{equation}
it is easy to see that $u^1$ satisfies the wave equation with this
potential:
\begin{equation}
\label{s1} u^1_{tt}-u^1_{xx}+Q(x)u^1=0,\quad x\geqslant
0,\,\,t\geqslant 0.
\end{equation}
Taking into account initial condition in (\ref{d1}) and the
equation $iu^1_{t}+u^2_{x}+qu^2=0$ at $t=0,$ we arrive at the
initial conditions
\begin{equation}
\label{s2} u^1(x,0)=u^1_t(x,0)=0.
\end{equation}
Counting the lat equality in (\ref{d1}), we get the boundary
condition
\begin{equation}
\label{s3} u^1(0,t)=f(t).
\end{equation}
Denote by $R_S$ the response operator (\ref{Response_sa}),
(\ref{react_rep}) for the wave equation (\ref{s1}), (\ref{s2}),
(\ref{s3}), and by $R_D$ the response operator (\ref{Dirac_react})
(\ref{Direc_react}) for the Dirac systems (\ref{d1}) with the
matrix potential (\ref{Spec_matr}). Everywhere below the
subscripts $S$ and $D$ being used refer the object to the
Schr\"odiner or Dirac system. For $R_S$ and $R_D$ we have by
(\ref{react_rep}) and (\ref{Direc_react}):
\begin{eqnarray*}
(R_S^f)(t)=(r_S*f)(t),\quad r_S(t)=-\delta'(t)+\widetilde r_S(s),\\
(R_Df)(t)=(r_D*f)(t),\quad r_D(t)=i\delta(t)+\widetilde r_D(t),
\end{eqnarray*}
here we separate singular and regular parts in integral kernels.
On the other hand, we can obtain from the Dirac system at $t=0$:
\begin{equation*}
u^1_x(0,t)=iu^2_t(0,t)+q(0)u^1(0,t)=i(R_Df)_t(t)+q(0)f(t)=i\left((R_Df)(t)\right)'.
\end{equation*}
Thus for arbitrary $f\in C_0^\infty(0,+\infty)$ we have that
\begin{equation*}
(R_Sf)(t)=(r_S*f)(t)=i(R_Df)'(t)=i(r_D'*f)(t).
\end{equation*}
The latter leads to the following relation between the kernels of
the response operators:
\begin{equation}
\label{resp_rav}
r_S(t)=ir'_D(t).
\end{equation}
The spectral representations (\ref{wave_repr}) and
(\ref{Dirac_repr}) implies
\begin{eqnarray}
\label{resp_sh_sp}
(R_Sf)(t)=u^f_x(0,t)=\int_{-\infty}^\infty\int_0^t\frac{\sin{\sqrt{\lambda}s}}{\sqrt{\lambda}}f(t-s)\,ds\,d\rho_S(\lambda),\\
(R_Df)(t)=u^2(0,t)=\int_{-\infty}^\infty\int_0^t e^{i\lambda
s}if(t-s)\,ds\,d\rho_D(\lambda).\label{resp_dir_sp}
\end{eqnarray}
Then from (\ref{resp_rav}), (\ref{resp_sh_sp}),
(\ref{resp_dir_sp}) follows the equality of the generalized
kernels of the response functions (see \cite{AM,MM3}):
\begin{equation*}
\int_{-\infty}^\infty-i\lambda e^{i\lambda
t}\,d\rho_D(\lambda)=\int_{-\infty}^\infty\frac{\sin{\sqrt{\lambda}t}}{\sqrt{\lambda}}\,d\rho_S(\lambda)
\end{equation*}
equating the real parts (the imaginary part in the l.h.s have to
be equal to zero), we get
\begin{equation*}
\int_{-\infty}^\infty\lambda \sin{\lambda
t}\,d\rho_D(\lambda)=\int_{-\infty}^\infty\frac{\sin{\sqrt{\lambda}t}}{\sqrt{\lambda}}\,d\rho_S(\lambda).
\end{equation*}
The latter equality yields
\begin{equation}
\label{Sh_Dir_mes}
\rho_S(\lambda)=\int_0^{\sqrt{\lambda}}\alpha^2\,d\rho_D(\alpha).
\end{equation}

How the De Branges spaces of Schr\"odinger and Dirac operators are
connected in our special situation? The De Branges spaces $B^T_D$,
$B^T_S$ corresponding Dirac and Scr\"odinger systems consist of
functions of the type (see (\ref{DB_Dir}), (\ref{DB2})):
\begin{eqnarray*}
F(\lambda)=\int_0^T\left(i\begin{pmatrix}e^{i\lambda
s}\\-e^{-i\lambda
s}\end{pmatrix},\begin{pmatrix}f(T-s)\\g(T-s)\end{pmatrix}\right)\,ds,\quad f,g\in L_2((0,T);\mathbb{C}),\\
G(\mu)=\int_0^T\frac{\sin{\sqrt{\mu}s}}{\sqrt{\mu}}h(T-s)\,ds,\quad
h\in L_2(0,T).
\end{eqnarray*}
Consider the subspace $B^T_s\subset B^T_D$, generated by the
vector functions of the special type:
$-\frac{1}{2}\begin{pmatrix}f\\f\end{pmatrix}$ with real-valued
$f\in L_2((0,T);\mathbb{R})$. In this case
\begin{equation}
\label{Spec_Direc} B^T_s:=\left\{ \int_0^T\sin{\lambda
s}f(T-s)\,ds\,\Big|\, f\in L_2(0,T)\right\}.
\end{equation}
We take $F\in B^T_s$ and evaluate the norm:
\begin{eqnarray*}
[F,F]_{B^T_D}=\frac{1}{4}\left(C^T_D\begin{pmatrix} f \\
f\end{pmatrix},\begin{pmatrix} f \\
f\end{pmatrix}\right)_{\mathcal{F}^T_D}\\
=\frac{1}{4}\int_{-\infty}^\infty
\overline{\int_0^T \left(\begin{pmatrix} ie^{i\lambda(T-s)} \\
-ie^{-i\lambda(T-s)}\end{pmatrix},\begin{pmatrix} f(s) \\ f(s)\end{pmatrix}\right)\,ds}{\int_0^T \left(\begin{pmatrix} ie^{i\lambda(T-t)} \\
-ie^{-i\lambda(T-t)}\end{pmatrix},\begin{pmatrix} f(t) \\
 f(t)\end{pmatrix}\right)\,ds}\, d\rho_D(\lambda)\\
= \int_{-\infty}^\infty\int_0^T\int_0^T {\sin{\lambda
s}f(T-s)}\sin{\lambda t}f(T-t)\,dt\,ds\,d\rho_D(\lambda)\\
= \int_{-\infty}^\infty\int_0^T\int_0^T {\frac{\sin{\sqrt{\lambda}
s}}{\sqrt{\lambda}}}f(T-s)\frac{\sin{\sqrt{\lambda}
t}}{\sqrt{\lambda}}f(T-t)\,dt\,ds\,d\rho_S(\lambda)=(C^T_Sf,f)_{\mathcal{F}^T_S}\\
=[\widetilde F,\widetilde F]_{B^T_S},\quad \widetilde
F(\lambda)=\int_0^T \frac{\sin{\sqrt{\lambda} s}}{\sqrt{\lambda}}
f(T-s)\,ds.
\end{eqnarray*}
Thus the Schr\"odinger De Branges $B^T_S$ of the system with the
potential (\ref{spec_potent}) are isometrically embedded into
Dirac De Branges space $B^T_D$ of the system with the matrix
potential (\ref{Spec_matr}) and $B^T_S$ is isometrically
isomorphic to the subspace $B^T_s$ of $B^T_D$ generated by the
functions of the special type (\ref{Spec_Direc}).

\section{Discrete Schr\"odinger operator}

For the real sequence $(b_n)$ we consider the discrete
Schr\"odinger operator in $l^2$ given by
\begin{equation}
\label{JM_oper} \left\{
\begin{array}l (H\phi)_n=\phi_{i+1}+\phi_{i-1}+b_n\phi_i,\quad n\geqslant 1,\\
(H\phi)_0=b_1\phi_0+\phi_1.
\end{array}
\right.
\end{equation}
Let $\varphi$ be the the solution to
\begin{equation}
\label{JM_spec_sol} \left\{
\begin{array}l \varphi_{i+1}+\varphi_{i-1}+b_n\varphi_i=z\varphi_i,\\
\varphi_0=0,\quad \varphi_1=1.
\end{array}
\right.
\end{equation}
We fix some $N\in \mathbb{N}$ and introduce the function
$E(z):=\varphi_N(z)-i\varphi_{N+1}(z)$ and show that it is a
Hermite-Biehler function. First we observe that
$\varphi_i(\overline z)=\overline{\varphi_i(z)}$. Then evaluating
$J_z$ in accordance with (\ref{repr_ker}):
\begin{eqnarray}
J_z(\xi)=\frac{\left(\overline{\varphi_N(z)}+i\overline{\varphi_{N+1}(z)}\right)\left({\varphi_N(\xi)}-i\varphi_{N+1}(\xi)\right)-
\left(\overline{\varphi_N(z)}-i\overline{\varphi_{N+1}(z)}\right)\left({\varphi_N(\xi)}+i\varphi_{N+1}(\xi)\right)}{2i(\overline
z-\xi)}\notag\\
=\frac{\left(\overline{\varphi_{N+1}(z)}\varphi_N(\xi)-{\varphi_{N+1}(\xi)}\overline{\varphi_N(z)}\right)}{\overline
z-\xi}.\label{JM_1}
\end{eqnarray}
Let us consider the equations
\begin{eqnarray*}
\overline{\varphi_{i+1}(z)}+\overline{\varphi_{i-1}(z)}+b_n\overline{\varphi_i(z)}=\overline{z}\overline{\varphi_i(z)},\\
{\varphi_{i+1}(\xi)}+{\varphi_{i-1}(\xi)}+b_n{\varphi_i(\xi)}={\xi}{\varphi_i(\xi)}.
\end{eqnarray*}
On multiplying first equation by $\varphi_i(\xi)$, second equation
by $\overline{\varphi_i(z)}$ and subtracting second from first, we
get
\begin{equation*}
\left(\overline{\varphi_{i+1}(z)}+\overline{\varphi_{i-1}(z)}\right)\varphi_i(\xi)-\left({\varphi_{i+1}(\xi)}+{\varphi_{i-1}(\xi)}\right)\overline{\varphi_i(z)}=\left(\overline
z -\xi\right)\overline{\varphi_i(z)}\varphi_i(\xi).
\end{equation*}
Summing up left and right hand sides of the previous equality from
$1$ to $N$, we get:
\begin{equation}
\label{JM_2} \left(\overline z
-\xi\right)\sum_{i=1}^N\overline{\varphi_i(z)}\varphi_i(\xi)=\varphi_{N+1}(z)\varphi_N(\xi)-{\varphi_{N+1}(\xi)}\overline{\varphi_N(z)}.
\end{equation}
Then from (\ref{JM_1}), (\ref{JM_2}) we see that
\begin{equation*}
J_z(\xi)=\sum_{i=1}^N\overline{\varphi_i(z)}\varphi_i(\xi),
\end{equation*}
and setting here $z=\xi$ we obtain
\begin{equation*}
0<\sum_{i=1}^N |\varphi_i(z)|^2=J_z(z)=\frac{|E(z)|^2-|E(\overline
z)|^2}{4\Im{z}}.
\end{equation*}
So $E$ is a Hermite-Biehler function. We can define De Branges
space $\widehat B^N_J$ based on this function. The opposite is
also true: if we have a De Branges space which comes from discrete
Schr\"odinger operator, one can recover corresponding canonical
system \cite{RR} by general technique \cite{DBr}.

For the same sequence $(b_n)$ we consider the dynamical system
with discrete time which is a natural analog of dynamical systems
governed by the wave equation with potential on a semi-axis:
\begin{equation}
\label{Jacobi_dyn} \left\{
\begin{array}l
u_{n,t+1}+u_{n,t-1}-u_{n+1,t}-u_{n-1,t}-b_nu_{n,t}=0,\quad n,t\in \mathbb{N}_0,\\
u_{n,-1}=u_{n,0}=0,\quad n\in \mathbb{N}, \\
u_{0,t}=f_t,\quad t\in \mathbb{N}_0.
\end{array}\right.
\end{equation}
By analogy with continuous problems, we treat the complex sequence
$f=(f_0,f_1,\ldots)\in \mathbb{C}^\infty$ as a \emph{boundary
control}. The solution to (\ref{Jacobi_dyn}) we denote by
$u^f_{n,t}$. In \cite{MM,MM1} the following representation have
been proved:
\begin{theorem}
The solution to (\ref{Jacobi_dyn}) admits the following
representation
\begin{equation}
\label{Jac_sol_rep} u^f_{n,t}=\prod_{k=0}^{n-1}
f_{t-n}+\sum_{s=n}^{t-1}w_{n,s}f_{t-s-1},\quad n,t\in
\mathbb{N}_0.
\end{equation}
where $w_{n,s}$ satisfies the Goursat problem
\begin{equation}
\label{Goursat} \left\{
\begin{array}l
w_{n,t+1}+w_{n,t-1}-w_{n+1,t}-w_{n-1,t}+b_nw_{n,t}=0,\quad n,s\in \mathbb{N}_0, \,\,s>n,\\
w_{n,n}=-\sum_{k=1}^n b_k,\quad n\in \mathbb{N},\\
w_{0,t}=0,\quad t\in \mathbb{N}_0.
\end{array}
\right.
\end{equation}
\end{theorem}
\begin{definition}
For $a,b\in l^\infty$ we define the convolution $c=a*b\in
l^\infty$ by the formula
\begin{equation*}
c_t=\sum_{s=0}^{t}a_sb_{t-s},\quad t\in \mathbb{N}
\end{equation*}
\end{definition}
By $\mathcal{F}^T$ we denote the \emph{outer space}, the space of
controls: $\mathcal{F}^T:=\mathbb{C}^T$, $f,g\in \mathcal{F}^T$,
$f=(f_0,\ldots,f_{T-1})$ with the inner product
$(f,g)_{\mathcal{F}^T}=\sum_{k=0}^{T-1}\overline{f_k}{g_k}$. As a
dynamical inverse data for (\ref{Jacobi_dyn}) we use the
\emph{response operator} which is a dynamical Dirichlet-to-Neumann
map: $R^T:\mathcal{F}^T\mapsto \mathcal{F}^T$ is defined by the
rule
\begin{equation*}
\left(R^Tf\right)_t=u^f_{1,t}, \quad t=1,\ldots,T.
\end{equation*}
By (\ref{Jac_sol_rep}):
\begin{eqnarray}
\label{R_def}
\left(R^Tf\right)_t=u^f_{1,t}=a_0f_{t-1}+\sum_{s=1}^{t-1}
w_{1,s}f_{t-1-s}
\quad t=1,\ldots,T.\\
\notag \left(R^Tf\right)=r*f_{\cdot-1}.
\end{eqnarray}
where the \emph{response vector} is the convolution kernel of the
response operator,
$r=(1,r_1,\ldots,r_{T-1})=(1,w_{1,1},w_{1,2},\ldots w_{1,T-1})$.

We introduce the \emph{inner space}, the space of states of the
dynamical system (\ref{Jacobi_dyn}) $\mathcal{H}^T:=\mathbb{C}^T$,
$h,l\in \mathcal{H}^T$, $h=(h_1,\ldots, h_T)$ with the inner
product $(h,l)_{\mathcal{H}^T}=\sum_{k=1}^T \overline{h_k}{l_k}$.
The \emph{control operator} $W^T:\mathcal{F}^T\mapsto
\mathcal{H}^T$ is defined by the rule
\begin{equation*}
W^Tf:=u^f_{n,T},\quad n=1,\ldots,T.
\end{equation*}
We notice that in \cite{MM,MM1} the authors used the real inner
space (and, consequently, the real outer space), but all the
results are valid for the complex controls as well. Everywhere
below, unless it is mentioned, we use the real outer and inner
spaces $\mathcal{F}^T$, $\mathcal{H}^T$. In \cite{MM} the authors
proved
\begin{theorem}
The control operator $W^T$ is an isomorphism between
$\mathcal{F}^T$ and $\mathcal{H}^T$.
\end{theorem}

According to \cite{AH,AT} there exist the spectral measure
$d\rho(\lambda)$ corresponding to (\ref{JM_oper}) with Dirichlet
condition $\phi_0=0$ such that for $u\in l^2$ the Fourier
transform $F: l^2\mapsto L_2(\mathbb{R},d\rho)$ is defined as
\begin{equation}
\label{JM_Four} (Fu)(\lambda)=\sum_{n=0}^\infty
u_k\varphi_k(\lambda)
\end{equation}
and the Parseval identity holds:
\begin{equation}
\label{JM_parseval} \sum_{k=0}^\infty u_kv_k=\int_{-\infty}^\infty
{(Fu)(\lambda)}(Fv)(\lambda)\,d\rho(\lambda).
\end{equation}
where
\begin{equation*}
u_k=\int_{-\infty}^\infty
(Fu)(\lambda)\varphi_k(\lambda)\,d\rho(\lambda)
\end{equation*}
Introduce the functions
\begin{equation*}
\left\{
\begin{array}l
\mathcal{T}_{t+1}+\mathcal{T}_{t-1}-\lambda_k \mathcal{T}_{t}=0,\\
\mathcal{T}_{0}=0,\,\, \mathcal{T}_1=1.
\end{array}
\right.
\end{equation*}
So $\mathcal{T}_k(2\lambda)$ are Chebyshev polynomials of the
second kind. In \cite{MM,MM1} the following spectral
representation for the solution to (\ref{Jacobi_dyn}) have been
derived:
\begin{equation}
\label{Jac_sol_spectr} u^f_{n,t}=\int_{-\infty}^\infty
\sum_{k=1}^t
\mathcal{T}_k(\lambda)f_{t-k}\varphi_n(\lambda)\,d\rho(\lambda)
\end{equation}
We put the following definition of the De Branges space,
associated with (\ref{JM_oper})
\begin{equation*}
B_J^T:=\left\{\left(Fu^f_{\cdot,T}\right)(\lambda)\,|\, f\in
\mathcal{F}^T\right\}.
\end{equation*}
We take $t=T$ in (\ref{Jac_sol_spectr}) and go over the Fourier
transform (\ref{JM_Four}). For real $\lambda$ we evaluate:
\begin{equation*}
\left(Fu^f_{\cdot,t}\right)(\lambda)=\sum_{n=0}^\infty
\int_{-\infty}^\infty \sum_{k=1}^T
\mathcal{T}_k(z)f_{T-k}\varphi_n(z)\,d\rho(z)\varphi_k(\lambda)=\sum_{k=1}^T
\mathcal{T}_k(\lambda)f_{T-k}.
\end{equation*}
Notice that for $\lambda\in \mathbb{C}$ we have the same formula
due to the analyticity of the l.h.s.  Thus we get the following
representation for $B_J^T$:
\begin{equation}
\label{DB2_JM} B_J^T:=\left\{\sum_{k=1}^T
\mathcal{T}_k(\lambda)f_{T-k}\,|\, f\in \mathcal{F}^T\right\}.
\end{equation}
The \emph{connecting operator}  $C^T: \mathcal{F}^T\mapsto
\mathcal{F}^T$ for (\ref{Jacobi_dyn}) is introduced via the
quadratic form:
\begin{equation}
\label{C_T_def} \left(C^T
f,g\right)_{\mathcal{F}^T}=\left(u^f_{\cdot,T},
u^g_{\cdot,T}\right)_{\mathcal{H}^T}=\left(W^Tf,W^Tg\right)_{\mathcal{H}^T},\,\,
C^T=\left(W^T\right)^*W^T.
\end{equation}
The fact that $C^T$ can be expressed in terms of the inverse data
is crucial in BC-method. The following theorem have been proved in
\cite{MM}:
\begin{theorem}
Connecting operator admits the representation in terms of
dynamical (response vector $r$) inverse data
\begin{equation}
\label{C_T_repr} C^T=C^T_{ij},\quad
C^T_{ij}=\sum_{k=0}^{T-\max{i,j}}r_{|i-j|+2k},\quad r_0=1.
\end{equation}
\begin{equation*}
C^T=
\begin{pmatrix}
r_0+r_2+\ldots+r_{2T-2} & r_1+r_3+\ldots+r_{2T-3} & \ldots &
r_T+r_{T-2} &
r_{T-1}\\
r_1+r_3+\ldots+r_{2T-3} & r_0+r_2+\ldots+r_{2T-4} & \ldots &
\ldots
&r_{T-2}\\
\cdot & \cdot & \cdot & \cdot & \cdot \\
r_{T-3}+r_{T-1}+r_{T+1} &\ldots & r_0+r_2+r_4 & r_1+r_3 & r_2\\
r_{T}+r_{T-2}&\ldots &r_1+r_3&1+r_2&r_1 \\
r_{T-1}& r_{T-2}& \ldots & r_1 &r_0
\end{pmatrix}
\end{equation*}
and spectral (spectral measure $d\rho$) inverse data:
\begin{equation}
\label{SP_mes_d} C^T_{l+1,m+1}=\int_{-\infty}^\infty
\mathcal{T}_{T-l}(\lambda)\mathcal{T}_{T-m}(\lambda)\,d\rho(\lambda),
\quad l,m=0,\ldots,T-1.
\end{equation}
and
\begin{equation*}
r_{k-1}=\int_{\infty}^\infty
\mathcal{T}_k(\lambda)\,d\rho(\lambda),\quad k\in \mathbb{N}.
\end{equation*}
\end{theorem}
In $B_J^T$ we introduce the scalar product by
\begin{equation}
\label{JM_scalprod}
[F,G]_{B^T_J}=\left(C^Tf,g\right)_{\mathcal{F}^T}.
\end{equation}
Since $C^T$ is a positive isomorphism, the space $B^T_J$ equipped
with the norm generated by (\ref{JM_scalprod}) is a Hilbert space.
We evaluate (\ref{JM_scalprod}) using (\ref{JM_parseval}):
\begin{equation*}
[F,G]_{B^T_J}=\left(u^f_{\cdot,T},u^g_{\cdot,T}\right)_{\mathcal{H}^T}=\int_{-\infty}^\infty
\overline{(Fu^f_{\cdot,T})(\lambda)}{(Fu^g_{\cdot,T})(\lambda)}\,d\rho(\lambda)=
\int_{-\infty}^\infty
\overline{F(\lambda)}{G(\lambda)}\,d\rho(\lambda).
\end{equation*}

We will be looking for the reproducing kernel in $B^T_J$ in the
form
\begin{equation}
\label{JM_reprker} J_z(\lambda)=\sum_{k=1}^T
\mathcal{T}_k(\lambda)j^z_{T-k},
\end{equation}
then by definition we should have for all $F\in B^T_J$ that
$[J_z,F]_{B^T_J}=F(z).$ The latter immediately implies that for
$z\in \mathbb{C}$
\begin{equation*}
[J_z,F]_{B^T_J}=\left(C^Tj^z,f\right)_{\mathcal{F}^T}=F(z)=\sum_{k=1}^T
\mathcal{T}_k(z)f_{T-k}=\left(\overline{\begin{pmatrix}\mathcal{T}_{T}(z)\\
\mathcal{T}_{T-1}(z)\\\cdot\\ \mathcal{T}_1(z)\end{pmatrix}},\begin{pmatrix}f_{0}\\
f_{1}\\\cdot\\ f_{T-1}\end{pmatrix}\right)_{\mathcal{F}^T}.
\end{equation*}
From where we get the following equation on $j^z$:
\begin{equation}
\label{JM_Krein} C^Tj^z=\overline{\begin{pmatrix}\mathcal{T}_{T}(z)\\
\mathcal{T}_{T-1}(z)\\\cdot\\ \mathcal{T}_1(z)\end{pmatrix}}.
\end{equation}

We set up the \emph{special control problem}: for $z\in
\mathbb{C}$ to find $j_z\in \mathcal{F}^T$ (specifically at this
point we need complex controls!) such that
\begin{equation}
\label{special_JM}
\left(W^Tj^z\right)_n=\varphi_n(z),\quad
n=1,\ldots,T.
\end{equation}
\begin{lemma}
The solution to the special control problem can be found as a
solution to (\ref{JM_Krein}).
\end{lemma}
\begin{proof}
We multiply (\ref{special_JM}) by $W^T g$, $g\in \mathcal{F}^T$ in
$\mathcal{H}^T$. As result we get that
\begin{equation}
\label{K7}
\left(C^Tj^z,g\right)_{\mathcal{F}^T}=\left(\varphi(\cdot,z),W^Tg\right)_{\mathcal{H}^T}=\overline{\sum_{n=1}^T\overline{\left(W^Tg\right)_n}\varphi_n(z)}.
\end{equation}
We evaluate the r.h.s. of the above equality using the spectral
representation (\ref{Jac_sol_spectr}):
\begin{eqnarray}
\overline{\sum_{n=1}^T\overline{\left(W^Tg\right)_n}\varphi_n(z)}=\overline{\overline{\sum_{n=1}^T\int_{-\infty}^\infty
\sum_{k=1}^T
\mathcal{T}_k(\lambda)g_{T-k}\varphi_n(\lambda)\,d\rho(\lambda)}\varphi_n(z)}\notag\\
=\sum_{k=1}^T \mathcal{T}_k(\lambda)g_{T-k}=\left(\overline{\begin{pmatrix}\mathcal{T}_{T}(z)\\
\mathcal{T}_{T-1}(z)\\\cdot\\ \mathcal{T}_1(z)\end{pmatrix}},\begin{pmatrix}g_{0}\\
g_{1}\\\cdot\\
g_{T-1}\end{pmatrix}\right)_{\mathcal{F}^T}.\label{K8}
\end{eqnarray}
From (\ref{K7})and (\ref{K8}) the statement of the lemma follows.
\end{proof}

The positivity of $C^T$ yields the function $E$ to be from
Hermite-Biehler class: from (\ref{JM_reprker}), (\ref{JM_Krein})
we easily get:
\begin{eqnarray*}
\frac{\left|E(z)\right|^2-\left|E(\overline
z)\right|^2}{4\operatorname{Im}z}=J_z(z)=\left(\overline{\begin{pmatrix}\mathcal{T}_{T}(z)\\
\mathcal{T}_{T-1}(z)\\\cdot\\ \mathcal{T}_1(z)\end{pmatrix}},\left(C^T\right)^{-1}\overline{\begin{pmatrix}\mathcal{T}_{T}(z)\\
\mathcal{T}_{T-1}(z)\\\cdot\\ \mathcal{T}_1(z)\end{pmatrix}}
\right)_{\mathcal{F}^T}>0
\end{eqnarray*}

For any positive $N$ we can consider the discrete Schr\"odinger
operator with some self-adjoint boundary condition at $n=N:$
\begin{equation}
\label{JM_spec_int} \left\{
\begin{array}l \varphi_{i+1}+\varphi_{i-1}+b_n\varphi_i=z\varphi_i,\\
\varphi_0=0,\quad \alpha\varphi_{N+1}+\beta\varphi_N=0.
\end{array}
\right.
\end{equation}
The (discrete) measure corresponding to (\ref{JM_spec_int}) we
denote by $d\rho_N(\lambda)$.
\begin{remark}
Due to the finite speed of propagation in the dynamical system
(\ref{Jacobi_dyn}), in all formulaes starting from spectral
representation of the solution (\ref{Jac_sol_spectr}), we can
substitute the measure $d\rho(\lambda)$ by any measure
$d\rho_N(\lambda)$ with $N> T.$ In particular
\begin{equation*}
[F,G]_{B^T_J}= \int_{-\infty}^\infty
\overline{F(\mu)}G(\mu)\,d\rho(\mu)= \int_{-\infty}^\infty
\overline{F(\mu)}G(\mu)\,d\rho_N(\mu)
\end{equation*}
\end{remark}
So, having constructed reproducing kernel $J_z$ by
(\ref{JM_Krein}), (\ref{JM_reprker}), by Theorem \ref{TeorDB} we
can recover the Hermite-Biehler function $E$, the space $B^T_J$ is
based on. Having in hands De Branges space $B^T_J$, one can
recover the underlying canonical system using the general approach
\cite{DBr}, or one can use the Boundary Control method for
discrete Schrodinger operator as it described in \cite{MM,MM1}.

\noindent{\bf Acknowledgments}

The research of Victor Mikhaylov was supported in part by NIR
SPbGU 11.38.263.2014 and RFBR 14-01-00535. Alexandr Mikhaylov was
supported by RFBR 14-01-00306; A. S. Mikhaylov and V. S. Mikhaylov
were partly supported by VW Foundation program "Modeling,
Analysis, and Approximation Theory toward application in
tomography and inverse problems." The authors are deeply indebted
to Prof. R.V. Romanov and Prof. M.I. Belishev for valuable
discussions.

\end{document}